\documentclass{article}
\usepackage[utf8]{inputenc}

\usepackage{amsmath, amsthm, amssymb, mathabx}
\usepackage{color}
\usepackage[all]{xy}
\usepackage{url}

\usepackage{graphicx}

\def\C{\mathbb C}    
\def\R{\mathbb R}  \def\Z{\mathbb Z}  
\def\D{\mathbb D}

\def\F{\mathcal F}  
\def\G{\Gamma}

\def\PSL{\operatorname{PSU}(1,1)} 
\def\Homeo{\operatorname{Homeo}_{+}(S^1)} 
\def\tHomeo{\widetilde{\operatorname{Homeo}}_{+}(S^1)} 
\def\ofg{\pi_{1}^{\operatorname{orb}}} 

\def\vol{\operatorname{vol}} 

\def\eu{\operatorname{eu}}
\def\eo{e^{\operatorname{orb}}}
\def\hol{\operatorname{hol}}
\def\m{\mathfrak{m}}
\def\pr{\operatorname{pr}}

\numberwithin{equation}{section}

\title{Harmonic measures and rigidity for transverse foliations on Seifert $3$-manifolds}
\author{
  Adachi, Masanori\footnote{Department of Mathematics, Faculty of Science, Shizuoka University, 836 Ohya, Suruga-ku, Shizuoka, 422-8529, Japan}\\
  \texttt{adachi.masanori@shizuoka.ac.jp}
  \and
  Matsuda, Yoshifumi\footnote{Department of Mathematical Sciences, College of Science and Engineering, Aoyama Gakuin University, 5-10-1 Fuchinobe, Chuo-ku, Sagamihara, Kanagawa, 252-5258,  Japan }\\
  \texttt{ymatsuda@math.aoyama.ac.jp}
  \and
  Nozawa, Hiraku\footnote{Department of Mathematical Sciences, Colleges of Science and Engineering, Ritsumeikan University, 1-1-1 Nojihigashi, Kusatsu, Shiga, 525-8577, Japan }\\
  \texttt{hnozawa@fc.ritsumei.ac.jp}
}
\date{}

\newtheorem{thm}{Theorem}[section]
\newtheorem{cor}[thm]{Corollary}
\newtheorem{lem}[thm]{Lemma}
\newtheorem{prop}[thm]{Proposition}

\theoremstyle{definition}
\newtheorem{defn}[thm]{Definition}

\theoremstyle{remark}
\newtheorem{rem}[thm]{Remark}

\usepackage{environ}

\NewEnviron{MSC}{%
    \footnotetext{\BODY}}[\par]

\NewEnviron{keywords}{%
    \footnotetext{\BODY}}[\par]

\begin{document}

\maketitle

\renewcommand*{\thefootnote}{}

\begin{MSC}
\end{MSC}


\begin{abstract}
Thurston proposed, in part of an unfinished manuscript, to study surface group actions on $S^1$ by using an $S^1$-connection on the suspension bundle obtained from a harmonic measure. Following the approach and previous work of the authors, we study the actions of general lattices of $\PSL$ on $S^1$. We prove the Gauss--Bonnet formula for the $S^1$-connection associated with a harmonic measure, and show that a harmonic measure on the suspension bundle of the action with maximal Euler number has rigidity, having a form closely related to the Poisson kernel. As an application, we prove a semiconjugacy rigidity for foliations with maximal Euler number, which is analogous to theorems due to Matsumoto, Minakawa and Burger--Iozzi--Wienhard.
\end{abstract}

\section{Introduction}
The Euler number is fundamental in the rigidity property of surface group actions on the circle $S^1$. Let $\G$ be a lattice in $\PSL$ and $\rho : \G \to \Homeo$ a homomorphism. In the case where $\G$ is torsion-free and uniform, the quotient $\Sigma = \G \backslash \D$ of the Poincar\'{e} disk $\D$ is a closed hyperbolic surface with $\pi_1 \Sigma \cong \G$. The Euler number $e(\rho)$ satisfies the Milnor--Wood inequality $|e(\rho)| \leq -e(\Sigma)$ \cite{Milnor,Wood}, and the equality holds if and only if $\rho$ is semi-conjugate to the Fuchsian action $\G \hookrightarrow \PSL \to \Homeo$ or its complex conjugate by Matsumoto's rigidity theorem \cite{Matsumoto1}.

If $\G$ is torsion-free and non-uniform, then the usual notion of the Euler number does not make sense since $\Sigma = \G \backslash \D$ is an open surface and the Euler number in the standard sense would be always trivial. Burger--Iozzi--Wienhard \cite{BIW}, however, extended the definition of the Euler number $e(\rho)$ to this case based on Ghys' bounded Euler class \cite{Ghys} and generalized the Milnor--Wood inequality and Matsumoto's rigidity theorem.
In our previous work \cite{AMN}, we obtained rigidity results for foliated harmonic measures of the suspension foliation $\F$ of $\rho$ based on the approach due to Frankel \cite{Frankel} and Thurston \cite{Thurston} to give alternative proofs to these preceding results. We used Thurston's connection on the suspension bundle associated with a harmonic measure, whose curvature $K$ satisfies $|K| \leq 1$. Note that the connection is merely continuous and its curvature is defined in a weak sense (see Definition \ref{def:conconn}).

In this article we extend these results for the general lattices in $\PSL$. First we will extend the definition of the Euler number $e(\rho)$ of the $\G$-action on $S^1$ to this general case following the manner of \cite{BIW}, and express $e(\rho)$ in terms of the translation numbers of the lifts of certain elements in $\rho(\G)$ to the universal cover $\tHomeo$ of $\Homeo$ (See Proposition \ref{prop:erho}). A key is the following construction of Thurston's connection and the Gauss--Bonnet formula in this context.

\begin{thm}\label{thm:GB}
Let $\G$ be a lattice in $\PSL$ with corresponding orbifold $\Sigma := \G \backslash \D$ and $\rho : \G \to \Homeo$ a homomorphism.  Assume that $\rho (\Gamma)$ has no finite orbit in $S^1$. Take a harmonic measure $\mu$ on the suspension foliation of $\rho$ on the suspension $S^1$-orbibundle $\Sigma \times_{\rho} S^1$. The $S^1$-connection form $\overline\omega$ associated with $\mu$ has curvature of the form $K \vol$ in a weak sense, where $\vol$ is the hyperbolic volume form on $\Sigma$ and $K$ is a measurable function such that $|K| \leq 1$, and we have
\[
e(\rho) = \frac{1}{2\pi} \int_{\Sigma} K(z) \vol(dz).
\]
\end{thm}

The Milnor--Wood inequality is a direct consequence of this result.

\begin{cor}\label{cor:MW}
For a lattice $\G$ in $\PSL$ and a homomorphism $\rho : \G \to \Homeo$, we have
\begin{equation}\label{eq:MW}
    |e(\rho)| \leq -\eo(\Sigma)
\end{equation}
where $\eo(\Sigma)$ is the orbifold Euler characteristic of $\Sigma = \G \backslash \D$ due to Satake.  
\end{cor}

Furthermore, with Theorem \ref{thm:GB}, we obtain the following rigidity for harmonic measures for lattice actions with maximal Euler numbers. 

\begin{thm}\label{thm:Poi}
Let $\G$ be a lattice in $\PSL$ and $\rho : \G \to \Homeo$ a homomorphism. If we have $e(\rho) = \eo(\Sigma)$, then the pull back of every harmonic measure on the suspension bundle of $\rho$ to the universal cover $\D \times \R$ is of the form 
\[
\frac{1-\lvert z \rvert^2}{\lvert \m(e^{\sqrt{-1}t})-z \rvert^2} \vol(z) \nu (t) \quad (z \in \D, t \in \R),
\]
where $\vol$ is the leafwise hyperbolic volume form, for a Borel measure $\nu$ on $\R$ and a continuous monotone mapping $\m : S^1 \to S^1$ of degree one which is $(\rho,\rho_0)$-equivariant, namely, $\m \circ \rho(\gamma) = \rho_0 (\gamma) \circ \m$ holds for any $\gamma \in \G$. 
\end{thm}

By this result, we can see that given lattice actions with maximal Euler numbers is semi-conjugated to the Fuchsian action by the mapping $\m$, which shows the following generalization of rigidity theorems of Matsumoto and Burger--Iozzi--Wienhard.

\begin{cor}\label{cor:M}
For a lattice $\G$ in $\PSL$ and a homomorphism $\rho : \G \to \Homeo$, we have $e(\rho) = \eo(\Sigma)$ $($resp.\ $e(\rho) = -\eo(\Sigma)$$)$ if and only if $\rho$ is semiconjugate to $\rho_0$ $($resp.\ $\overline{\rho}_{0}$$)$, where $\rho_{0}$ is the Fuchsian action of $\G$ induced from the inclusion $\G \to \PSL \to \Homeo$, and $\overline{\rho}_{0}$ is its complex conjugate. 
\end{cor}

We remark that Minakawa \cite{Minakawa} has already proved a generalization of the Milnor--Wood inequality and Matsumoto's rigidity theorem for uniform lattices of $\PSL$. 
 By Selberg's lemma \cite{Selberg}, any lattice $\G$ has a torsion-free finite index normal subgroup $\Pi$. Minakawa derived his result for uniform lattice $\G$ from Matsumoto's theorem for uniform torsion-free $\Pi$. Similarly we can prove Corollary \ref{cor:M} for general lattice $\G$ by using the result for torsion-free $\Pi$. The proof of Corollary \ref{cor:M} by Theorem \ref{thm:Poi}, therefore, should read as a direct alternative proof for this fact based on foliated harmonic measure. We note that Mann--Wolff \cite{Mann-Wolff} studied the action of the mapping class groups of surfaces on the circle by using the orbifold fundamental groups of some $2$-orbifolds, which is the main object in this article.

For lattices whose quotient orbifold have nonzero genus, we can refine the Milnor--Wood inequality in terms of the Seifert invariants by using a result of Eisenbud--Hirsch--Neumann \cite[Theorem 3.6]{EHN} for transverse foliations on compact Seifert manifolds with boundary. Note that, for lattices with nontrivial torsion, the suspension bundle of its action on $S^1$ may not be a $3$-manifold, but only a Seifert $3$-orbifold (see Section \ref{sec:Eu} for the detail and the Seifert invariants). 

\begin{thm}\label{thm:EHN}
Let $\G$ be a lattice in $\PSL$ with corresponding orbifold $\Sigma := \G \backslash \D$ and $\rho : \G \to \Homeo$ a homomorphism. Let $M$ be the suspension bundle of $\rho$ whose the Seifert invariant of $M$ normalized with respect to $\rho$ is $(g;(1,\beta_0), (\alpha_1,\beta_1), \dots, (\alpha_\ell,\beta_\ell))$. If the topological genus $g$ of $\Sigma$ is not $0$, then we have
\begin{equation}\label{eq:EHN0}
\eo(\Sigma) + \sum_{i=1}^{\ell}\frac{(\alpha_i - 1) - \beta_i}{\alpha_i} \leq e(\rho) \leq - \eo(\Sigma) - \sum_{i=1}^{\ell}\frac{\beta_i - 1}{\alpha_i}.
\end{equation}
\end{thm}

There is a clear relation between the equality conditions of this inequality and the Milnor--Wood inequality as follows:

\begin{thm}\label{thm:EHNeq}
Let $\G$ be a lattice in $\PSL$ with corresponding orbifold $\Sigma := \G \backslash \D$ and $\rho : \G \to \Homeo$ a homomorphism. We have $e(\rho) = \eo(\Sigma)$ if and only if we have the equality in the left inequality in \eqref{eq:EHN0} and $\beta_i = \alpha_i - 1$ for every $i = 1, \dots, \ell$. We have $e(\rho) = -\eo(\Sigma)$ if and only if we have the equality in the right inequality in \eqref{eq:EHN0} and $\beta_i = 1$ for every $i = 1, \dots, \ell$. 
\end{thm}

In the exceptional case where the topological genus of $\Sigma$ is $0$, by \cite[Theorem 3.6]{EHN}, we have
\begin{equation*}
\eo(\Sigma) -1+ \sum_{i=1}^{\ell}\frac{(\alpha_i - 1) - \beta_i}{\alpha_i} \leq e(\rho) \leq - \eo(\Sigma) +1- \sum_{i=1}^{\ell}\frac{\beta_i - 1}{\alpha_i},
\end{equation*}
which can be weaker than the Milnor--Wood inequality.

Let us mention an application of Corollaries \ref{cor:MW} and \ref{cor:M} to transverse foliations on Seifert manifolds $M$ whose base orbifold $\Sigma$ possibly has cusps. There is a well known correspondence between homomorphisms $\pi_1^{\operatorname{orb}}\Sigma \to \Homeo$ and foliations on Seifert manifolds transverse to fibers via suspension bundles. Then we can define the Euler number $e(\F)$ of a transverse foliation $\F$ as the Euler number of the corresponding homomorphism $\pi_1^{\operatorname{orb}}\Sigma \to \Homeo$. Note that, if $\Sigma$ has no cusp, then $e(\F)$ coincides with the Euler number of the Seifert manifold $M$. The following is a rigidity result for Seifert manifolds with transverse foliations, which is equivalent to Corollaries \ref{cor:MW} and \ref{cor:M} via the correspondence.

\begin{cor} 
Let $M$ be a Seifert $3$-orbifold whose singular locus is the union of a finite number of $S^1$-fibers. Assume that the base orbifold $\Sigma$ is an oriented hyperbolic surface of finite volume without boundary that possibly has cusps. 
If $M$ admits a foliation $\F$ transverse to the Seifert fibration, then we have 
\[
|e(\F)| \leq -\eo(\Sigma),
\]
where $e(\F)$ is the Euler number of the action of the orbifold fundamental group of $\Sigma$ on $S^1$ given by $\F$. 

Moreover, the equality holds if and only if there exist a lattice $\G$ in $\PSL$ and a homotopy equivalence $f \colon M \to \G \backslash \PSL$ which monotonely maps each Seifert fiber to a Seifert fiber and whose restriction to each leaf of $\F$ is a diffeomorphism to a leaf of the stable or unstable foliation on $\G \backslash \PSL$. 
\end{cor}

The stable (resp.\ unstable) foliation on $\G \backslash \PSL$ is the orbit foliation of the right action of the upper (resp.\  lower) triangular matrix subgroup of $\PSL$. A closed Seifert manifold with hyperbolic base space and nonzero Euler number admits an $\widetilde{SL}(2;\R)$-geometry \cite{Scott}, but it may not be a homogeneous space of $\PSL$ in general. Closed Seifert manifolds which are homogeneous spaces were characterized by \cite{RaymondVasquez}.

\paragraph{Acknowledgments.}
M.A.\ is supported by JSPS KAKENHI Grant Numbers JP19KK0347, JP21H00980, JP21K18579, JP23K20793 and JP24K06776.
 H.N.\ is supported by JSPS KAKENHI Grant Numbers JP20K03620 and JP24K06723.

\paragraph{Convention.}
We use the identification $S^1 = \{ z \in \C \mid |z| = 1 \} = \R / 2\pi\Z$ throughout this paper.

\section{Harmonic measures on the suspension foliation}

Transverse invariant measures are a basic tool to study foliations. However, there are many foliations without such good measures. Garnett \cite{Garnett} introduced a weaker notion, harmonic measures on a foliated manifold, which are measures invariant under the leafwise heat flow. The advantage is that a nontrivial harmonic measure always exists for compact foliated manifolds by Garnett's theorem. We refer to \cite{Garnett,Candel} for the basics of harmonic measures and to \cite{Yue,DK,Adachi2} for their applications. 

Let $\G$ be a lattice of $\PSL$ and $\Sigma :=\G\backslash \D$ the quotient orbifold. For a homomorphism $\rho : \G \to \Homeo$, we can construct its suspension $M := \Sigma \times_\rho S^1 := \G\backslash (\D \times S^1) \to \Sigma$, where $\gamma \cdot (z,t) = (\gamma z,\rho(\gamma)t)$ for $\gamma \in \G, z\in \D, t\in S^1$. We have a canonical map $\pi : M \to \Sigma; [(z,t)] \mapsto [z]$.
Since $\G$ may have torsions, the $\G$-actions on $\D$ and $\D \times S^1$ may not be free but only proper. Therefore, in general, $\Sigma$ and $M$ are not smooth manifolds but orbifolds. Similarly $\pi$ may have multiple fibers. A partition $\F$ on $M$ is induced from the product foliation $\D \times S^1 = \sqcup_{w \in S^1}\D \times \{w\}$. If $\G$ is torsion-free, then $\F$ is the suspension foliation. In general, the image of $\D \times \{w\}$ may be an orbifold, and $\F$ is a foliation by orbifolds, which means that $(M,\F)$ is covered by an orbifold atlas consisting of foliated charts with finite group actions and compatible foliations.
Smooth functions on $\Sigma$ and $M$ are identified with $\G$-invariant smooth functions on $\D$ and $\D \times S^1$, respectively. $C^k$ functions along the leaves of $\F$ are identified with $\G$-invariant continuous functions on $\D \times S^1$ such that its restriction to $\D \times \{w\}$ is $C^k$ for each $w \in S^1$. Differential forms, Riemannian metrics and Laplacians on these orbifolds are considered in a similar way by using the $\G$-actions on the corresponding spaces. Since $\G$ preserves the Poincar\'e metric on $\D$, the leaves of $\F$ admits a natural Riemannian metric $g$ of constant curvature $-1$. Let us denote the Laplacian associated with $g$ by $\Delta$.

The following is the main tool in this article.
\begin{defn}
A Borel measure $\mu$ on $M$ is called \emph{harmonic} if 
\[
\int_M \Delta f(x) \mu(dx) = 0 
\]
for every compactly supported $C^2$ function $f$ along the leaves of $\F$ such that $\Delta f$ is continuous on $M$.
\end{defn}

Alvarez \cite{Alvarez} proved the existence of a harmonic measure for the suspension bundle of a $\G$-action for torsion-free $\G$ even if the total space may not be compact. Thanks to the Selberg's lemma \cite{Selberg}, a general lattice of $\PSL$ has a torsion-free normal subgroup of finite index. Therefore, by using a finite covering of $M$, we get the following by the result of Alvarez:

\begin{lem}\label{lem:Alvarez}
There exists a harmonic probability measure $\mu$ on the suspension foliation $(M,\F)$.
\end{lem}

By the disintegration formula due to \cite{
Garnett}, the pull back of $\mu$ to $\D \times \R$ is of the form 
\begin{equation}\label{eq:disint}
\tilde{\mu} = h(z,t) \vol(z) \nu (t),
\end{equation}
where $\vol$ is the leafwise volume measure on $\D$, for some Borel measurable function $h(z,t)$ whose restriction to $\D \times \{t\}$ is a positive harmonic function for $\nu$-a.e.\ $t$ and $\nu$ is a Borel measure on $\R$.

\section{The Euler number of actions of lattices on $S^1$} \label{sec:def_euler}
\subsection{Definition of the Euler number}
Let $\G$ be a lattice of $\PSL$, and $\rho : \G \to \Homeo$ a homomorphism. In the case where $\G$ is uniform, it is classical to define the Euler number of $\rho$ by the pairing of the Euler class and the fundamental class of $\G\backslash \D$. Burger--Iozzi--Wienhard \cite{BIW} defined the Euler number based on bounded cohomology in the case where $\G$ is non-uniform and torsion-free, i.e., $\G\backslash \D$ is a hyperbolic surface with cusps. Let us extend their definition to the case where $\G$ is a general lattice. 

By Selberg's lemma \cite{Selberg}, there exists a finite index torsion-free normal subgroup $\Pi$ of $\G$. Let $\Sigma := \G\backslash \D$. Let $H_b(\cdot)$ denote the bounded cohomology with real coefficient of a topological space or a discrete group. First let us explain that so-called Gromov isomorphism for $\G$ follows from that for $\Pi$ \cite{Gromov}.

\begin{lem}
We have an isomorphism
\begin{equation}\label{eq:Gromov}
H_b^{k}(\G) \cong H_b^{k}(\Sigma), \quad\quad \forall k \geq 0.
\end{equation}
\end{lem}
\begin{proof} 
Since $\G/\Pi$ is a finite group, its bounded cohomology with any coefficient module vanish by a result of Trauber (see \cite{Gromov}). Then the Lyndon--Hochschild--Serre spectral sequence \cite{Lyndon,HS} for bounded cohomology for the sequence $1 \to \Pi \to \G \to \G/\Pi \to 1$, 
\[
H_b^{p}(\G/\Pi,H^{q}_b(\Pi))\Longrightarrow H^{p+q}_b(\G)
\]
collapses at the $E^2$-level, which implies
\[  
H_b^{k}(\G) \cong H^{0}_b(\G/\Pi,H^{k}_b(\Pi)) \cong H^{k}_b(\Pi)^{\G/\Pi}.
\]
Let $B^{\bullet}(\cdot)$ be the complex of bounded singular cochains on a topological space. By a result of Ivanov \cite{Ivanov}, we have $H_b^{k}(\Pi) \cong H^{k}(B^{\bullet}(\D)^{\Pi})$ for every $k \geq 0$, where $B^{\bullet}(\D)^{\Pi}$ denotes the subspace of $\Pi$-invariant elements. Therefore, since $\Pi$ is torsion-free and $\G/\Pi$ is a finite group, we have
\[
H^{k}_b(\Pi)^{\G/\Pi} \cong  H^{k}((B^{\bullet}(\D)^{\Pi})^{\G/\Pi}) \cong H^{k}(B^{\bullet}(\Pi\backslash \D)^{\G/\Pi}) \cong H^{k}_b(\Sigma),
\]
which implies \eqref{eq:Gromov}.
\end{proof}

Let us define the Euler number of a homomorphism $\rho : \G \to \Homeo$ following \cite{BIW}.
We cut off all cusps of $\Sigma$ and obtain a compact orbifold $\Sigma'$ with boundary $\partial\Sigma'$. We can pull back the universal Euler class $\eu \in H_b^{2}(\Homeo)$ by $\rho$ to have $\rho^*\eu \in H_b^{2}(\G)$, which we regard as an element of $H_b^{2}(\Sigma')$ via the isomorphism \eqref{eq:Gromov}. Since $H_b^{1}(\partial \Sigma')\cong H_b^{2}(\partial \Sigma')\cong 0$, by the relative exact sequence, we have an isomorphism
\[
f : H_b^{2}(\Sigma', \partial \Sigma') \to H_b^{2}(\Sigma').  
\]
\begin{defn}
We define the Euler number $e (\rho) \in \R$ of $\rho$ by
\[
e (\rho) = \langle f^{-1}\rho^*\eu, [\Sigma',\partial \Sigma'] \rangle,
\]
where $[\Sigma',\partial \Sigma']$ is the fundamental class of $(\Sigma',\partial \Sigma')$.
\end{defn}

By the formula due to Burger--Iozzi--Wienhard \cite{BIW}, if $\G$ is torsion-free and non-uniform, we may express the Euler number $e(\rho)$ in terms of translation numbers of the holonomy along cusps. Let
\[
\tHomeo = \{ \, f \in \operatorname{Homeo}_+(\R) \mid f(x+2\pi) = f(x) + 2\pi, \forall x \in \R \, \}
\]
be the universal cover of $\Homeo$ and $\tau : \tHomeo \to \R$ denotes the translation number, which is defined by
\[
\tau (f) = \frac{1}{2\pi} \lim_{n \to \infty} \frac{f^{n}(x) - x}{n}
\]
for $f \in \tHomeo$. Then, by \cite{BIW}, we have
\begin{equation}\label{eq:BIW1}
e (\rho) = -\sum_{j=1}^{m} \tau(\tilde{\rho}(c_j)),
\end{equation}
where $c_j$ is a loop on $\Sigma$ that goes around the $j$-th cusp $(j = 1, \dots, m)$ oriented in a way compatible with the boundary of $\Sigma$ when we cut off the cusps, and a homomorphism $\tilde{\rho} : \G \to \tHomeo$ is a lift of $\rho$.

\subsection{The Seifert invariant and the translation number}\label{sec:Eu}

In this section, let us quickly recall the classical notion of the Seifert invariants, and introduce a normalized version for Seifert manifolds with transverse foliation. By using this normalized invariant, we will extend this formula \eqref{eq:BIW1} to the case where $\G$ has a nontrivial torsion taking in account of multiple fibers of the suspension bundle $\pi : M \to \Sigma$ in the next section. Since $\G$ is assumed to contain a torsion, the $\G$-actions on $\D$ and $\D \times S^1$ may not be free but only proper in general. Therefore $\Sigma$ is not a hyperbolic surface but a hyperbolic $2$-orbifold, and $M$ is not a circle bundle but a Seifert fibered $3$-orbifold in general.

Let us recall the Seifert invariant of a multiple fiber of a Seifert $3$-orbifold (see, for example, \cite{Audin,BS}). A multiple fiber has a neighborhood of the form $B^2 \times_{\Z/\alpha \Z} S^1$, where $B^2$ is a $2$-disk and the $(\Z/\alpha \Z)$-action on $B^2 \times S^1$ is generated by $(v,t) \mapsto (e^{2\pi \sqrt{-1}/\alpha} v, e^{-2\beta \pi \sqrt{-1}/\alpha}t)$ for some positive integer $\alpha$ and an integer $\beta$. Note that $\beta$ is well defined only up to modulo $\alpha$. Taking $\beta$ such that $0 \leq \beta < \alpha$, the pair $(\alpha,\beta)$ is called \emph{the Seifert invariant of the multiple fiber}, which classifies the local normal form of the fiber.

\begin{rem}
Let us see that the normal form of multiple fibers in our situation differ from the case of Seifert $3$-manifolds at most only up to orbifold singular loci. Let $\delta = \operatorname{GCD}(\alpha,\beta)$ and $\alpha'=\alpha/\delta$, $\beta'=\beta/\delta$. Since the action of the subgroup $H < \Z/\alpha \Z$ of order $\delta$ on $B^2 \times S^1$ is generated by $(v,t) \mapsto (e^{2\pi \sqrt{-1}/\delta}v, t)$, it is easy to see that the isotropy subgroup of $\{ 0 \} \times S^1$ is $H$. By using a homeomorphism $B^2/H \cong B^2$, we can see that the quotient $(B^2/H) \times S^1$ with the induced $\Z/\alpha' \Z$-action is equivariantly homeomorphic to $B^2 \times S^1$ with the $\Z/\alpha' \Z$-action generated by $(v,t) \mapsto (e^{2\pi \sqrt{-1}/\alpha'} v, e^{-2\beta' \pi \sqrt{-1}/\alpha'}t)$. The latter is the normal form of the multiple fiber of a Seifert $3$-manifold with Seifert invariant $(\alpha',\beta')$. Thus, ignoring the orbifold singularity on the multiple fiber $\{0\} \times_{\Z/\alpha \Z} S^1$, the local model is topologically same as the case of multiple fibers of Seifert $3$-manifolds.
\end{rem}

\begin{rem}
 If $M$ is a Seifert $3$-manifold, then $\alpha$ is coprime to $\beta$. Thus the Seifert invariant $(\alpha,\beta)$ is recovered from $\beta/\alpha$. In our case, $\alpha$ may not be coprime with $\beta_i$, and $\operatorname{GCD}(\alpha,\beta)$ is determined by the orbifold singularity on the multiple fiber.
\end{rem}

As in the classical case of Seifert $3$-manifolds, Seifert $3$-orbifolds with transverse foliation over surfaces with cusps are topologically (relative to the trivialization near cusps) classified with Seifert invariants. Let us introduce a normalized version, which we call the Seifert invariant normalized with respect to the given transverse foliation or the given $\G$-action on $S^1$. 
Let $O_1, \dots, O_\ell$ be the multiple $S^1$-fibers of $M$ and $(\alpha_i, \beta_i)$ the Seifert invariant of $O_i$ for each $i = 1, \dots, \ell$. Take a free $S^1$-fiber $O_0$. For $i=0, \dots, \ell$, let $D_i$ be a small open disk on $\Sigma$ around $\pi (O_i)$. Let $\Sigma'' = \Sigma - \cup_{i=0}^\ell D_i$. Take a trivialization $\sigma$ of $S^1$-bundle $M|_{\Sigma''}$ over $\Sigma''$. Note that we take $O_0$ so that we can trivialize $M|_{\Sigma''}$ even in the case where $M$ has no multiple fibers. Let $\rho''$ denote the $S^1$-action $\pi_1(\Sigma'') \to \ofg(\Sigma) \overset{\rho}{\to} \Homeo$ which is the composite of $\rho$ and the map $\pi_1(\Sigma'') \to \ofg(\Sigma)$ induced by the inclusion $\Sigma'' \to \Sigma$. By using the trivialization, we can lift $\rho''$ to an $\R$-action $\tilde{\rho}'' \colon \pi_1(\Sigma'') \to \tHomeo$. Indeed, let us consider the bundle map $\varphi : \Sigma'' \times \R \to M|_{\Sigma''}$ whose restriction to each fiber is $\R \to \R/2\pi\Z=S^1$ and which maps $\Sigma'' \times \{0\}$ to the image of $\sigma$. For every closed path $c$ on $\Sigma''$, we can consider the lift of the holonomy of the flat $S^1$-bundle $M|_{\Sigma''}$. 

Let us explain that there is a canonical way to choose the homotopy class of the trivialization of $M|_{\Sigma''}$. First let us consider how to trivialize $M|_{\Sigma''}$ near multiple fibers. 
In each torus $\pi^{-1}(\partial D_i)$, $i =1, \dots,  \ell$, consider the homology classes of the meridian $a_i$ of the solid torus, any of the free $S^1$-fibers $b_i$ and the trivialization $\sigma (\partial D_i)$. It is well known that the homology class $[a_i]$ is expressed as
\begin{equation}\label{eq:betai}
[a_i] = \alpha_i [\sigma(\partial D_i)] - \hat{\beta}_i [b_i], 
\end{equation}
where an integer $\hat{\beta}_i$ such that $\hat{\beta}_i \equiv \beta_i \mod \alpha_i$ (see, e.g., \cite[Sec.\ I.3.c]{Audin}). It is easy to see that we can change $\hat{\beta}_i$ to $\hat{\beta}_i+s\alpha_i$ for any $s \in \Z$ by changing the homotopy class of $\sigma (\partial D_i)$, and there is a unique homotopy class of $\sigma (\partial D_i)$ such that $\hat{\beta}_i = \beta_i$.
Note that the translation number of the lift of the holonomy around the multiple fiber is determined by the homotopy class of the trivialization on $M|_{\partial D_i} \to \partial D_i$. By these facts, we can normalize the translation number of the lift of the holonomy around the multiple fiber as follows. 

\begin{lem}[see e.g., \cite{BS}]\label{lem:GB0}
Let $d_i$ be a loop on $\Sigma$ which goes around the $i$-th cone point $\pi(O_i)$ in the positive direction with respect to the orientation induced from $\Sigma$ for $i= 1, \dots, \ell$. For a trivialization of $M|_{\partial D_i} \to \partial D_i$ unique up to homotopy, we have
\[
\tau(\tilde{\rho}'' (d_i)) = \frac{\beta_i}{\alpha_i}
\]
for $i=1,\dots,\ell$, where $\widetilde{\rho}''(d_i)$ is the lift of $\rho''(d_i)$ with respect to the trivialization. 
\end{lem}

\begin{proof}
Let $V_i = \pi^{-1}(D_i)$ for $i=1,\dots, \ell$. We can assume that $V_i$ is of the normal form $V_i=B^2 \times_{\Z/\alpha_i \Z} S^1$, where the $(\Z/\alpha_i \Z)$-action on $\hat{V}_i=B^2 \times S^1$ is generated by the map $(v,t) \mapsto (e^{2\pi \sqrt{-1}/\alpha_i} v, e^{-2\beta_i \pi \sqrt{-1}/\alpha_i}t)$. It is easy to see that we can assume that the lift of the flat connection $\omega$ on $V_i$ to $\hat{V}_i$ is equal to $\hat{\omega}$, which is the flat connection form on $\hat{V}_i$ whose integral manifolds are the fibers of the projection $\hat{V}_i = B^2 \times S^1 \to S^1$. Then, for $i=1,\dots, \ell$, by \eqref{eq:betai}, we have 
\[
[\hat{a}_i] = [\sigma (\partial D_i)] - \hat{\beta}_i [\hat{b}_i] 
\]
for some integer $\hat{\beta}_i$ with $\hat{\beta}_i \equiv \beta_i \mod \alpha_i$, where $\hat{a}_i$ is the meridian of $\hat{V}_i$ and $\hat{b}_i$ is the curve $\{0\} \times S^1$ (see \eqref{eq:betai}). Thus the meridian $\hat{a}_i$ goes around the $S^1$-fiber $-\hat{\beta}_i$ times with respect to the trivialization $\sigma (\partial D_i)$. Therefore there exists a trivialization of $M|_{\partial D_i} \to \partial D_i$, which is unique up to homotopy, such that the meridian $\hat{a}_i$ goes around the $S^1$-fiber $-\beta_i$ times. Since $\hat{a}_i$ is a concatenation of $\alpha_i$ lifts of $a_i$ and the orientation of $d_i$ is opposite to that of $\partial D_i$, we have $\tau (\tilde{\rho}''(d_i)^{\alpha_i}) = {\beta}_i$. Therefore
 we have
\[
\tau (\tilde{\rho}''(d_i)) = \frac{\tau (\tilde{\rho}''(d_i)^{\alpha_i})}{\alpha_i} = \frac{\beta_i}{\alpha_i}.\qedhere
\]
\end{proof}

Let us consider how to trivialize $M|_{\Sigma''}$ near the cusps.  Let $c_j$ be a loop on $\Sigma$ which goes around the $j$-th cusp in the positive direction with respect to the orientation induced from $\Sigma$ for $j= 1, \dots, m$. As in Lemma \ref{lem:GB0}, we can change $\tau(\tilde{\rho}''(c_j))$ by any integer by changing the homotopy class of the trivialization of $M|_{\Sigma''}$ near the cusp, and there is a unique homotopy class of the trivialization near the $j$-th cusp such that $0 \leq \tau(\tilde{\rho}''(c_j)) < 1$. Let us denote the decimal part of the translation number of the lifts of $\rho(c_j)$ by $\tau_{\operatorname{dec}}(\rho(c_j))$, which does not depend on the lift being determined by $\rho(c_j)$. 

Let us define the normalized version of the Seifert invariant. As we saw in Lemma \ref{lem:GB0} and the last paragraph, there is, unique up to homotopy, a trivialization of the $S^1$-bundle $M|_{\Sigma''}$ over the union of $\partial D_1, \dots, \partial D_\ell$ and neighborhoods of each cusp such that we have
\begin{align*}
0\leq \tau(\tilde{\rho}''(d_i)) <1 & \quad (i=1,\dots,\ell), \\
0\leq \tau(\tilde{\rho}''(c_j))<1 & \quad (j= 1, \dots, m).
\end{align*}
Note that we have $\tau(\tilde{\rho}''(d_i))=\frac{\beta_i}{\alpha_i}$ for $i=1,\dots,\ell$. It is well known that there exists a trivialization $\sigma$ of $\pi$ over $\partial D_0$, unique up to homotopy, such that the trivialization of $\pi$ on the union of $\partial D_0, \partial D_1, \dots, \partial D_\ell$ and neighborhoods of each cusp extends to $\Sigma - \cup_{i=0}^{\ell}D_i$. Let $\beta_0$ be the integer such that
\[
[a_0] = [\sigma(\partial D_0)] - \beta_0 [b_0]
\]
in $H_1(\pi^{-1}(\partial D_0))$, where $a_0$ is the meridian of $\pi^{-1}(D_0)$ and $b_0$ is any free $S^1$-fiber. Let $g$ be the genus of $\Sigma$. Following the terminology in the case of Seifert $3$-manifolds \cite{EHN}, we will call the data 
$(g;(1,\beta_0), (\alpha_1,\beta_1), \dots, (\alpha_\ell,\beta_\ell))$ \emph{the Seifert invariant of the suspension bundle $M$ normalized with respect to $\rho$}. 

\begin{rem}
Note that we adopt the sign convention on $\beta_0$ same as \cite{EHN} and different from \cite{Audin}. From the point of view of the Euler number, our convention gives a simpler formula \eqref{eq:Eulernumber}.
\end{rem}

\subsection{The Euler number and the translation number}

In this section, let us generalize the formula of the Euler number \eqref{eq:BIW1} to the case where the lattice $\G$ may not be torsion-free. We fix an action $\rho : \G \to \Homeo$ and assume that the suspension bundle $M$ of $\rho$ has Seifert invariant $(g;(1,\beta_0), (\alpha_1,\beta_1), \dots, (\alpha_\ell,\beta_\ell))$ normalized to $\rho$. For a subsurface $A$ of $\Sigma$, let $e(M|_{A})$ denote the Euler number of the flat bundle over $A$ obtained by the restriction of $M \to \Sigma$. Due to \cite{BIW}, in the case where $\G$ is torsion-free, the Euler number satisfies the following sum formula 
\[
e(\rho) = e(M|_{\Sigma_0}) + e(M|_{\Sigma_1}),
\]
where $\Sigma_0$ and $\Sigma_1$ are connected subsurfaces of $\Sigma$ such that $\Sigma_0 \cup \Sigma_1 = \Sigma$ and $\Sigma_0 \cap \Sigma_1$ is the union of a finite number of $S^1$. It is easy to see that this sum formula holds more generally in our situation where $\G$ may have torsion. 
Applying this formula to the disks around cone points, we get a generalization of the formula \eqref{eq:BIW1}. 

\begin{prop}\label{prop:erho} 
The Euler number of $\rho$ is given by
\begin{equation}\label{eq:Eulernumber}
e (\rho) = - \beta_0 -\sum_{i=1}^{\ell}\frac{\beta_i}{\alpha_i}-\sum_{j=1}^{m} \tau_{\operatorname{dec}} (\rho(c_j)),
\end{equation}
where $c_j$ is a loop on $\Sigma$ which goes around the $j$-th cusp in the positive direction with respect to the orientation induced from $\Sigma$ for $j= 1, \dots, m$.
\end{prop}

\begin{proof}
Recall the notation $\Sigma'' = \Sigma - \cup_{i=0}^{\ell} D_i$. By the sum formula and the triviality of $M|_{D_i}$, we have 
\[
e(\rho) = e(M|_{\Sigma''}) + \sum_{i=0}^{\ell} e(M|_{D_i}) = e(M|_{\Sigma''}).
\]
Trivialize $M|_{\Sigma''} \to \Sigma''$ so that $0\leq \tau(\tilde{\rho}''(d_i)) <1$ for $i=1,\dots,\ell$ and $0\leq \tau(\tilde{\rho}''(c_j))<1$ and for $j= 1, \dots, m$ as in the definition of the Seifert invariant normalized with respect to $\rho$. Since $M|_{\Sigma''}$ has no multiple fibers, by \eqref{eq:BIW1} and Lemma \ref{lem:GB0}, we have
\[
e(M|_{\Sigma''}) = -\sum_{i=0}^{\ell} \tau (\tilde{\rho}''(d_i)) -\sum_{j=1}^{m} \tau_{\operatorname{dec}}(\rho(c_j)) = - \beta_0 -\sum_{i=1}^{\ell}\frac{\beta_i}{\alpha_i}-\sum_{j=1}^{m} \tau_{\operatorname{dec}} (\rho(c_j)),
\]
which implies \eqref{eq:Eulernumber}. \qedhere

\end{proof}

In order to prove the Gauss--Bonnet formula (Theorem \ref{thm:GB}), we need to consider $S^1$-connections on $M$ which are only continuous. Fix an $S^1$-action whose lift to $\D \times S^1$ is free and has the orbits tangent to the second $S^1$-factor. Let $\eta$ be an $S^1$-invariant $1$-form on $M$ such that $\eta(Z) = 1$, where $Z$ is the infinitesimal generator of the $S^1$-action. This $\eta$ plays the role of the connection form.

\begin{defn}\label{def:conconn} 
Let $\Omega$ be an integrable $2$-form on $\Sigma$. For a (non-smooth) continuous $S^1$-connection form $\eta$ on $M$, we call $\Omega$ the \emph{curvature} of $\eta$ if, for every simply-connected subset $U$ of $\Sigma$ which does not contain cone points of $\Sigma$ and whose boundary $\partial U$ is piecewise smooth, we have $\hol_\eta(\partial U) = \exp\left( \sqrt{-1} \int_{U} \Omega\right)$, where $\hol_\eta(\partial U) \in S^1$ is the holonomy of $\eta$ along $\partial U$. 
\end{defn}

Let us consider the lift $\tilde{\eta} := \varphi^*\eta$ of the connection form $\eta$, where $\varphi \colon \Sigma'' \times \R \to M|_{\Sigma''}$ is the covering map considered in the definition of the Seifert invariant normalized with respect to $\rho$. Then the lift of the holonomy of $\eta$ is equal to the holonomy of $\tilde{\eta}$. Since $\eta$ is an $S^1$-connection, the holonomy of $\tilde{\eta}$ along a path $c$ is a translation, whose translation number is computed with any horizontal lift of $c$: For a path $c$ on $\Sigma''$ which may not be closed, we have 
\[
\tau(\widetilde{\hol}_{\eta}(c))= \frac{\pr_2(\tilde{c}(1)) - \pr_2(\tilde{c}(0))}{2\pi},
\]
where $\pr_2 \colon \Sigma'' \times \R \to \R$ is the second projection and $\tilde{c}$ is a horizontal lift of $c$ with respect to $\tilde{\eta}$. By the same argument as \cite[Lemma 2.5]{AMN}, we have the following. 

\begin{lem}\label{lem:GB}
Let $\eta$ be a continuous $S^1$-connection form on $M$. We cut off all cusps of $\Sigma$ and disks $D_0, \dots, D_\ell$ to regard $\Sigma'' = \Sigma - \cup_{i=0}^{\ell} D_i$ as a compact surface with boundary loops $c'_1, \dots, c'_m$ and $\partial D_0, \dots, \partial D_\ell$. If $\eta$ has curvature $\Omega$ in the sense of Definition \ref{def:conconn}, then we have
\begin{equation}\label{eq:GB}
\frac{1}{2\pi} \int_{\Sigma''}\Omega = - \sum_{i=0}^{\ell} \tau(\widetilde{\hol}_{\eta}(\partial D_i)) + \sum_{j=1}^{m} \tau(\widetilde{\hol}_{\eta}(c'_j)).
\end{equation}
\end{lem}

\section{The Gauss--Bonnet formula for the connection associated with a harmonic measure}\label{sec:proof}

\subsection{Construction of the $S^1$-connection}\label{conn_sec}

We can construct the $S^1$-connection $\overline\omega$ associated with a harmonic measure in a way similar to the torsion-free case \cite{AMN}. Let $\G$ be a lattice of $\PSL$, $\rho$ a homomorphism $\G \to \Homeo$ and $\Sigma = \G \backslash \D$. We assume that $\rho (\Gamma)$ has no finite orbit in $S^1$. 
Consider the suspension foliation on $M=\Sigma \times_{\rho} S^1$ equipped with the leafwise Poincar\'e metric. There exists a harmonic probability measure $\mu$ on this foliation by Lemma \ref{lem:Alvarez}. 
Consider the lift $\tilde\mu$ of $\mu$ to $\D \times \R$. By disintegration formula \eqref{eq:disint}, $\tilde\mu$ is of the form
\[
\tilde{\mu} = q(z,t) \vol(z) \nu(t), 
\]
where $\vol$ is the leafwise volume measure on $\D$, for some Borel measure $\nu$ on $\R$ and a locally integrable function $q$ on $\D \times \R$ with respect to $\vol(z) \nu(t)$ such that $q(\cdot, t)$ is a positive harmonic function on $\D$ for $\nu$--a.e.\ $t$. 

The first step is to collapse the complement of the support of $\mu$. For each $z \in \D$, consider a measure $\mu_z$ on $S^1$ induced by $\tilde{\mu}_z := q(z,t)\nu(t)$ on $\R$. Since there is no finite orbit of $\rho(\G)$ in $S^1$, every leaf is a hyperbolic good orbifold of infinite volume. Since such orbifold admits no $L^1$ harmonic function by \cite[Theorem 2.4]{Li-Schoen} or \cite[Theorem 1]{Li}, it follows that $\mu_z$ is non-atomic for every $z \in \D$. We consider a non-decreasing map $\tilde{\psi}: \R \to \R$ defined by 
\[
\tilde{\psi}(t) := \int_0^t \tilde\mu_0(ds)  = \int_0^t q(0,s) \nu(ds).
\]
Since $\mu_0$ is non-atomic, $\tilde{\psi}$ is continuous. It covers a semiconjugacy $\psi : S^1 \to S^1$ which collapses the complement of the support of $\mu_0$. Thus we obtain a $\G$-action $\rho'$ such that $\rho'(\gamma) \circ \psi = \psi \circ \rho (\gamma)$ for every $\gamma \in \G$. Let $M'=\Sigma \times_{\rho'} S^1$ be the suspension bundle of $\rho'$, which is a Seifert $3$-orbifold in general. The suspension foliation of $M'$ admits a harmonic measure $\mu'$ of full support induced by $\mu$, whose lift $\tilde{\mu}'$ to $\D \times \R$ is expressed as 
\[
\tilde{\mu}' = h(z,t) \vol(z) \lambda(t), 
\]
where $\lambda$ is the Lebesgue measure on $\R$, for some locally integrable function $h$ on $\D \times \R$ with respect to $\vol(z) \lambda(t)$ such that $h(\cdot, t)$ is a positive harmonic function on $\D$ for $\lambda$--a.e.\ $t$.

The second step is the construction of a smooth structure on $M'$ such that $\rho'$ is bi-Lipschitz. For this purpose, we integrate the harmonic measure $\mu'$ on each fiber of $\pi: M' \to \Sigma$. Namely, we consider the map $\tilde{\Phi}: \D \times \R \to \D \times \R$ given by $\tilde{\Phi}(z,t) = (z, \varphi(z,t))$ where
\begin{equation}\label{varphi_eq}
\varphi(z,t):=\int_0^t h(z,s) ds.
\end{equation}
Then, by Harnack's inequality, we can show that $\tilde{\Phi}$ is a locally bi-Lipschitz homeomorphism. Since $h(z,t)$ is harmonic in $z$, so is $\varphi$. In particular, it is smooth in $z$. By conjugating $\rho'$ on $\D \times \R$ by $\tilde\Phi$, we get a smooth $\G$-action on $\D \times \R$, which induces a smooth structure on $M'$. Since $\tilde\Phi$ is locally bi-Lipschitz, the $\G$-action is bi-Lipschitz.

In order to show that $\varphi$ descends to a map $\D \times S^1 \to \D \times S^1$, we need the fact that the total measure on each free $S^1$-fiber is constant. This follows from the fact that the function $z \mapsto \mu'_z(S^1) = \int_0^{2\pi} h(z,t) dt$ is positive harmonic and invariant under the $\G$-action: Indeed, if it is not constant, it induces a non-constant $L^1$ positive harmonic function on $\Sigma$, which contradicts \cite[Theorem 2.4]{Li-Schoen} or \cite[Theorem 1]{Li}.

Finally let us construct $\overline\omega$ by taking the fiberwise average of the flat connection. Let $\omega$ be the flat connection form defining the suspension foliation of $\rho'$. By taking the average of $\omega$ under the principal $S^1$-action $\{\phi_t \}$ that preserves $\mu_z'$ for a.e.\ $z \in \Sigma$, we obtain an $S^1$-connection of  $M'$, \begin{equation}\label{eq:omegabar}
\overline\omega := \frac{1}{2\pi}\int_0^{2\pi} \phi^*_t\omega\, dt,
\end{equation}
which we call the $S^1$-\textit{connection form associated with} $\mu$. 

\subsection{Proof of the Gauss--Bonnet formula}

Let us prove Theorem \ref{thm:GB} by using the notation of the last section \ref{conn_sec}. The former half of this proof is the same as in \cite{AMN}, but we include the outline for the convenience of the readers. We can assume that the given harmonic measure $\mu$ is of full support on the suspension bundle $M$ by collapsing the complement of $\operatorname{supp} \mu$ as in the third paragraph of this section; We replace $M$, $\rho$ and $\mu$ with $M'$, $\rho'$ and $\mu'$, respectively.

Let us consider the $S^1$-connection $\overline{\omega}$ associated with $\mu$ on the suspension bundle $M$. In the following local computation, we will locally identify $\Sigma$ with $\D$. First let us show that $\overline{\omega}$ has a curvature form in the sense of Definition \ref{def:conconn}. For each $z = x_1 + \sqrt{-1} x_2 \in \D$, let $\tau(z, \cdot)$ denote the inverse map of $\varphi(z, \cdot): \R \to \R$  given by \eqref{varphi_eq}. Let 
\begin{equation}\label{eq:omegaj}
\omega_j(z, \theta) := \left.\frac{\partial \varphi}{\partial x_j}(z, t)\right\rvert_{t = \tau(z,\theta)}
\end{equation}
for $j = 1, 2$. Since $\omega$ is Lipschitz, by Harnack's inequality, we can show that the function $\omega_j$ is continuous on $\D \times S^1$.
For each $z \in \D$, $\omega_j(z, \cdot)$ is a periodic Lipschitz function on $\R$ of period $2\pi$.
Then we have
\[
\overline{\omega} = d\theta - \sum_{j=1,2}\left( \int_0^{2\pi} \omega_j(z, \theta) \frac{d\theta}{2\pi} \right) dx_j
\]
on $\D \times S^1$. Since $\omega$ is Lipschitz, by the argument in \cite{Adachi1,AMN}, we can show $\overline{\omega}$ has the curvature of the form $K(z) \vol(dz)$ in the sense of Definition \ref{def:conconn}, where $K$ is a measurable function on $\Sigma$ given by
\begin{equation}\label{K_eq}
K(z) := -\frac{(1-\lvert z \rvert^2)^2}{4}\int_0^{2\pi} \left(-\frac{\partial \omega_1}{\partial \theta}(z, \theta) \omega_2(z, \theta) + \frac{\partial \omega_2}{\partial \theta}(z, \theta) \omega_1(z, \theta) \right) \frac{d\theta}{2\pi}.
\end{equation}
By applying the isoperimetric inequality and Harnack's inequality, we will show that $\lvert K(z) \rvert$ is bounded by one everywhere.
Among the definition of $K(z)$ in \eqref{K_eq},
\[
\int_0^{2\pi} \frac{1}{2}\left(-\frac{\partial \omega_1}{\partial \theta}(z, \theta) \omega_2(z, \theta) + \frac{\partial \omega_2}{\partial \theta}(z, \theta) \omega_1(z, \theta) \right) d\theta
\]
is the signed area of the domain bounded by a Lipschitz closed curve $(\omega_1(z,\theta), \omega_2(z,\theta))$ $(0 \leq \theta \leq 2\pi)$ on $\R^2$. The tangent vector of this curve is defined for a.e.\ $\theta$ and we have
\begin{align*}
\frac{\partial}{\partial \theta}\omega_j(z, \theta) = \frac{\partial \log h}{\partial x_j}(z, \tau(z,\theta))
\end{align*}
by a straight forward computation. Then, by the isoperimetric inequality, we have
\begin{multline}\label{eq:Kest}
\lvert K \rvert \leq \frac{(1-\lvert z \rvert^2)^2}{4\pi} \cdot \frac{1}{4\pi} \left(\int_0^{2\pi} \sqrt{\left(\frac{\partial \log h}{\partial x_1}\right)^2 + \left(\frac{\partial \log h}{\partial x_2}\right)^2} d\theta\right)^2 \\
= \frac{1}{4\pi^2} \left(\int_0^{2\pi} \| d \log h \| d\theta\right)^2,
\end{multline}
which implies $\lvert K \rvert \leq (2\pi)^2/{4\pi^2} = 1$ by Harnack's inequality.

The next step is to show the Gauss--Bonnet formula for $\overline{\omega}$
\begin{equation}\label{eq:GB34}
e(\rho) = \frac{1}{2\pi} \int_{\Sigma} K(z) \vol (dz).
\end{equation}
This is a theorem of Satake \cite{Satake} when $\Sigma$ is a closed orbifold. Let us consider the case where $\Sigma$ has cusps. By hypothesis, the suspension bundle $M \to \Sigma$ has multiple fibers $O_1, \dots, O_\ell$ whose Seifert invariants are $(\alpha_1,\beta_1), \dots, (\alpha_\ell,\beta_\ell)$, respectively. 
We take a free $S^1$-fiber $O_0$, and fix a trivialization $\sigma$ of $M$ on $\Sigma - \cup_{i=0}^{\ell}\pi(O_i)$. Then, by Proposition \ref{prop:erho}, we have
\begin{equation}\label{eq:tr1}
e (\rho) = - \beta_0 -\sum_{i=1}^{\ell}\frac{\beta_i}{\alpha_i}-\sum_{j=1}^{m} \tau_{\operatorname{dec}} (\rho(c_j)),
\end{equation}
where $c_j \in \Gamma$ corresponds to a loop that goes around the $j$-th cusp of $\Sigma$ for each $j=1, \dots, m$. 
A neighborhood of each cusp is foliated by closed horocircles. For each $j$, let $\{c_j^s\}_{s \gg 0}$ be a family of horocircles in a neighborhood of the $j$-th cusp so that
the hyperbolic diameter $\delta^s_j$ of $c_j^s$ tends to $0$   
and $c_j^s$ approaches to the $j$-th cusp as $s \to \infty$. We also consider a family of disks $D_i^{s}$ in $\Sigma$ centered at $\pi(O_i)$ for $i = 0, \dots, \ell$ such that the area of $D_i^s$ goes to $0$ and the length of $\partial D_i^s$ goes to $0$ as $s \to \infty$. 
For $s \gg 0$, let $\Sigma^s$ be the compact subsurface of $\Sigma - \cup_{i=0}^{\ell}D_i^s$ bounded by $c_1^s, \dots, c_m^s$ and $\partial D_0^s, \dots, \partial D_\ell^s$ (see Fig.\ \ref{fig:1}).  
        \begin{figure}[ht]
	    	\centering
		    \includegraphics[scale=0.5]{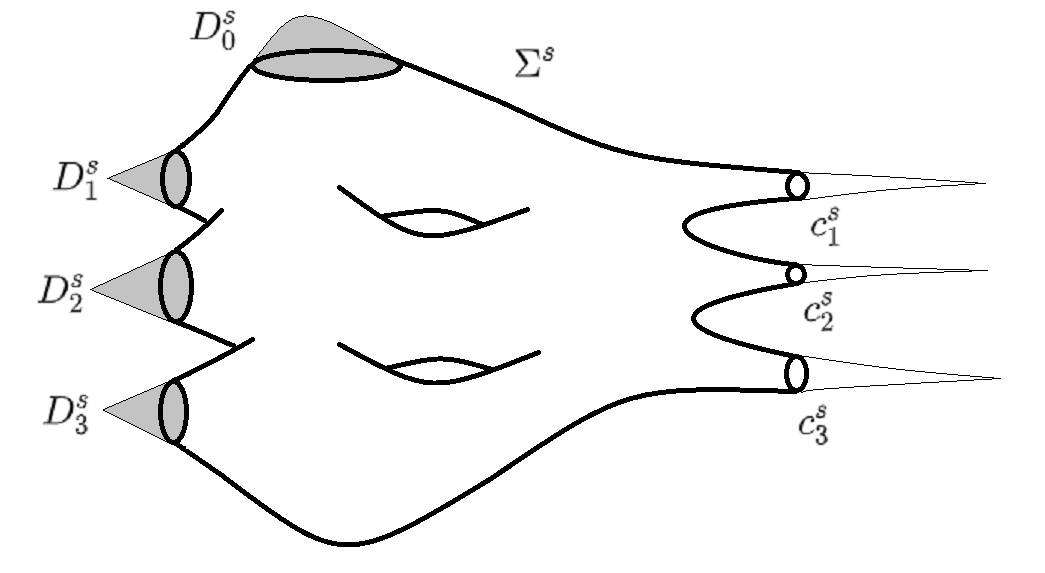}
    		\caption{Subsurface $\Sigma^s$}
	    	\label{fig:1}
        \end{figure}
For $s \gg 0$, by Lemma \ref{lem:GB}, we have
\begin{equation}\label{eq:tr2}
 \frac{1}{2\pi} \int_{\Sigma^s}K(z) \vol (dz)= -\sum_{i=0}^{\ell} \tau(\widetilde{\hol}_{\overline\omega}(\partial D_i^s)) + \sum_{j=1}^{m} \tau(\widetilde{\hol}_{\overline\omega}(c_j^s)),
\end{equation}
where $\widetilde{\hol}_{\overline\omega}(\gamma) : \R \to \R$ is the lift of the holonomy map of $\overline\omega$ along a curve $\gamma$ with respect to $\sigma$. Since $\Gamma \backslash \D$ has finite volume and $\lvert K(z) \rvert \leq 1$ a.e.\ $z$ by the last step, we have  $\int_{\Sigma^s} K(z) \vol(dz) \to \int_\Sigma K(z) \vol(dz)$ as $s \to \infty$. 
Therefore, by \eqref{eq:tr1} and \eqref{eq:tr2}, for the proof of the Gauss--Bonnet formula \eqref{eq:GB34}, it remains to find a sequence $\{s_n\}$ such that 
\begin{align}\label{eq:eachcusp1}
\tau(\widetilde\hol_{\overline\omega}(\partial D_0^{s_n})) & \longrightarrow \beta_0, \\
\tau(\widetilde\hol_{\overline\omega}(\partial D_i^{s_n})) & \longrightarrow \frac{\beta_i}{\alpha_i}, \label{eq:eachcusp2} \\
\tau(\widetilde\hol_{\overline\omega}(c_j^{s_n})) & \longrightarrow -\tau_{\operatorname{dec}} (\rho(c_j)) \label{eq:eachcusp3}
\end{align}
as $n \to \infty$ for each fixed $1\leq i\leq n$ and $1\leq j \leq m$. This can be shown in the same way as in \cite[\S3.4]{AMN} thanks to Harnack's inequality. In order to prove \eqref{eq:eachcusp1} and \eqref{eq:eachcusp2}, note that we can follow the computation in \cite[\S3.4]{AMN} only with the assumption that the area of $D_i^s$ goes to $0$ and the length of $\partial D_i^s$ goes to $0$ as $s \to \infty$ without having a cusp at the center of $D_i^s$. Note that all translation numbers in \eqref{eq:eachcusp1},\eqref{eq:eachcusp2} and \eqref{eq:eachcusp3} are computed for the lifts with respect to the trivialization $\sigma$ that was used to define the Seifert invariant normalized with respect to $\rho$. 

\subsection{Applications to rigidity}

A key point of the application of the Gauss--Bonnet formula (Theorem \ref{thm:GB}) to rigidity is the characterization of harmonic functions on the Poincar\'e disk $\D$ which achieves the equality in Harnack's inequality: If $h$ is a positive harmonic function on $\D$ such that $\|d \log h\| = 1$ at some point, then $h$ is of the form $h(z) =  C\frac{1-\lvert z \rvert^2}{\lvert m-z \rvert^2}$ for some $m \in S^1$ and a positive constant $C$ (see \cite[Lemma 2.3]{AMN}). We outline a direct proof of the rigidity theorem (Theorem \ref{thm:Poi}) for harmonic measures based on Theorem \ref{thm:GB} below, which is essentially contained in \cite{AMN}.

\begin{proof}[Proof of Theorem \ref{thm:Poi}]
Recall that $\tilde\mu$ is of the form
\[
\tilde{\mu} = q(z,t) \vol(z) \nu(t), 
\]
where $\vol$ is the leafwise volume measure on $\D$, for some Borel measure $\nu$ on $\R$ and a locally integrable function $q$ on $\D \times \R$ with respect to $\vol(z) \nu(t)$ such that $q(\cdot, t)$ is a positive harmonic function on $\D$ for $\nu$--a.e.\ $t$. Let $\tilde{\psi} : \R \to \R$ be the map that collapses the complement of the support of $\tilde{\mu}_0 = q(0,t) \nu$, which we considered in Section \ref{conn_sec}. Then, the pushforward $\tilde{\mu}' = (\operatorname{id} \times \psi)_* \tilde{\mu}$ is of the form
\[
\tilde{\mu}' = h(z,t) \vol(z) \lambda(t), 
\]
where $\lambda$ is the Lebesgue measure and $h(\cdot,t)$ is a positive harmonic function on $\D$ for $\lambda$--a.e.\ $t$. Since $\tilde{\psi}_*(q(0,t)\nu) = \lambda$, we have
\begin{equation}\label{eq:qzt}
q(z,t) = q(0,t) h(z,\tilde{\psi} (t))
\end{equation}
for $\nu$--a.e.\ $t$ in $\R$. By the hypothesis $e(\rho)=\eo(\Sigma)$, by the Gauss--Bonnet theorem $\frac{1}{2\pi}\int_{\Sigma} K(z) \vol(dz) = \eo(\Sigma)$ for orbifolds due to Satake \cite{Satake} the equalities hold in the isoperimetric inequality for the curve $(\omega_1(z,\theta),\omega_2(z,\theta))$ (see \eqref{eq:omegaj}) and Harnack's inequality for $h(\cdot,t)$ in \eqref{eq:Kest}. Since $h(\cdot, t)$ achieves the equality in Harnack's inequality, by \cite[Lemma 2.3]{AMN}, we have that, for $\lambda$--a.e.\ $t$ in $\R$, we have
\begin{equation*}\label{eq:harmonicmeas}
h(z,t) =  \frac{1-\lvert z \rvert^2}{\lvert m(e^{\sqrt{-1}t})-z \rvert^2}
\end{equation*}
for some $m(e^{\sqrt{-1}t}) \in S^1$. Regard $m$ as a map $S^1 \to S^1$. From the equality in the isoperimetric inequality, it follows that $(\omega_1(z,\theta),\omega_2(z,\theta))$ is a round circle for a.e.\ $z \in \D$, which implies that $m$ is a rotation. Thus we have 
\[
h(z,t) =  \frac{1-\lvert z \rvert^2}{\lvert e^{\sqrt{-1}(t+t_0)} -z \rvert^2},
\]
for some $t_0 \in \R$. By \eqref{eq:qzt}, we have 
\[
q(z,t) = q(0,t) \frac{1-\lvert z \rvert^2}{\lvert \mathfrak{m}(e^{\sqrt{-1}t}) -z \rvert^2},
\]
where $\mathfrak{m}(e^{\sqrt{-1}t}) = e^{\sqrt{-1}(\tilde{\psi}(t) + t_0)}$, for $\nu$--a.e.\ $t$ in $\R$. Then $\mathfrak{m}$ is a continuous degree one monotone map. The $(\rho,\rho_0)$-equivariance of $\mathfrak{m}$ follows from the $\G$-invariance of $\tilde{\mu}$. By replacing the original $\nu$ with $q(0,t) \nu$, we have Theorem \ref{thm:Poi}. 
\end{proof}

\section{Eisenbud--Hirsch--Neumann's inequality}

Let us prove Theorems \ref{thm:EHN} and \ref{thm:EHNeq}. Recall the hypothesis and the notation. Given a lattice $\G$ of $\PSL$ and a homomorphism $\rho : \G \to \Homeo$, the suspension bundle is a Seifert $3$-orbifold $M$ (see Section \ref{sec:Eu}). Take a free $S^1$-fiber $O_0$, and let $O_1, \dots, O_{\ell}$ be the multiple fibers of $M$. Let $(g;(1,\beta_0), (\alpha_1,\beta_1), \dots, (\alpha_\ell,\beta_\ell))$ be the Seifert invariant of $M$ normalized with respect to $\rho$ (see Section \ref{sec:Eu}). We assume that the topological genus $g$ of the base orbifold $\Sigma := \G \backslash \D$ is not $0$. Recall that $\pi : M \to \Sigma$ is trivialized on $\Sigma - \cup_{i=0}^{\ell} D_i$, where $D_i$ is a small open disk neighborhood of $\pi(O_i)$ for $i=0, \dots, \ell$.

Note that the orbifold locus of $M$ is contained in the union of multiple fibers. Then after cutting off an $S^1$-invariant tubular neighborhood of each component of the singular locus, the remaining part can be regarded as the interior of a compact Seifert $3$-manifold with boundary. Since $g \geq 1$, by a result of Eisenbud--Hirsch--Neumann \cite[Theorem 3.6]{EHN}, we have 
\begin{equation}\label{eq:EHNb}
2-2g-\ell - \sum_{j=1}^m \lceil \underline{m}h_j \rceil \leq \beta_0 \leq 2g-2 -\sum_{j=1}^m \lfloor \overline{m}h_j \rfloor,
\end{equation}
where $m$ is the number of cusps of the base orbifold $\Sigma$, the map $h_j \colon \R \to \R$ is the lift of the holonomy along the curve that goes around the $j$-th cusp with respect to the fixed trivialization, $\underline{m}h = \min_{x} (h(x)-x)$ and $\overline{m}h = \max_{x} (h(x)-x)$ for $h \in \tHomeo$. We use the following inequality.

\begin{lem}\label{lem:tr}
For $h \in \tHomeo$, we have
\begin{equation*}
    \lceil \overline{m}h \rceil - 1 \leq \tau (h) \leq \lfloor \underline{m}h \rfloor +1.
\end{equation*}
\end{lem}

\begin{proof}
Take $n \in \Z$  so that $n \leq \tau(h) < n+1$. Assume that $n+1 \leq \overline{m}h$. Since $\tau(h) < n+1$, we have $h(x) - x < n+1$ for some $x \in \R$. By the intermediate value theorem, we have $h(y) - y = n+1$ for some $y \in \R$. Since $h \in \tHomeo$, we have $h^s(y) - y = s(n+1)$, which implies that  $\tau(h) = n+1$. It contradicts to $\tau(h) < n+1$. Therefore, by contradiction, we have $\overline{m}h < n+1$, which implies $\lceil \overline{m}h \rceil - 1 \leq \tau (h)$. The second inequality is proved in a similar way.
\end{proof}

Theorem \ref{thm:EHN}, a refinement of the Milnor--Wood inequality, follows from \eqref{eq:EHNb} and Lemma \ref{lem:tr}.

\begin{proof}[Proof of Theorem \ref{thm:EHN}]
By \eqref{eq:EHNb} and Lemma \ref{lem:tr}, we have
\begin{align}
0 & \leq -2 + 2g + m - \sum_{j=1}^m \tau(h_j) - \beta_0 \quad\quad \text{and} \label{eq:EHN1} \\
0 & \leq -2 + 2g + m + \sum_{j=1}^m \tau(h_j) + \ell + \beta_0. \label{eq:EHN2} 
\end{align}
Since $e(\rho) = - \beta_0 - \sum_{i=1}^{\ell}\frac{\beta_i}{\alpha_i} - \sum_{j=1}^m \tau(h_j)$ by Proposition \ref{prop:erho} and it is well known that $\eo(\Sigma) = 2 - 2g - m - \sum_{i=1}^{\ell}\frac{\alpha_i - 1}{\alpha_i}$ (see, e.g., \cite[Eq.\ 13.3.4]{Thurston4}), we have the inequality \eqref{eq:EHN0}.
\end{proof}

Theorem \ref{thm:EHNeq} is immediately proved by using the same notation as follows.

\begin{proof}[Proof of Theorem \ref{thm:EHNeq}]
By definition of Seifert invariants, we have 
\begin{align}\label{eq:ab}
 0 & \leq \frac{(\alpha_{i}-1)-\beta_i}{\alpha_i}, & 0 & \leq \frac{\beta_{i}-1}{\alpha_i}
\end{align}
for each $i=1,\dots,\ell$. Therefore, we have $e(\rho) = \eo(\Sigma)$ if and only if we have the equality in the first inequality in \eqref{eq:EHN0} and $\beta_i = \alpha_i - 1$ for every $i = 1, \dots, \ell$. Similarly we have $e(\rho) = -\eo(\Sigma)$ if and only if we have the equality in the second inequality in \eqref{eq:EHN0} and $\beta_i = 1$ for every $i = 1, \dots, \ell$. 
\end{proof}

\end{document}